\documentclass[10pt]{amsart}
\usepackage{amsmath}
\usepackage{amsfonts}
\usepackage{amssymb}
\usepackage{amsthm}
 \usepackage[foot]{amsaddr}
\usepackage{url}
\usepackage[all]{xy}
\usepackage{dsfont}
\usepackage{graphicx}
\usepackage{caption}
\usepackage{subcaption}
\usepackage{comment} 
\usepackage{stmaryrd}
\usepackage{hyperref}
\usepackage{todonotes}
\usepackage{color}
\usepackage{enumerate,tikz-cd}
\usepackage{fixltx2e}
\usepackage{MnSymbol}
\allowdisplaybreaks

\numberwithin{equation}{section}

\newtheorem{proposition}{Proposition}[section]
\newtheorem{lemma}[proposition]{Lemma}
\newtheorem{theorem}[proposition]{Theorem}
\newtheorem{corollary}[proposition]{Corollary}

\theoremstyle{definition}
\newtheorem{remark}[proposition]{Remark}
\newtheorem{definition}[proposition]{Definition}
\newtheorem{example}[proposition]{Example}

\renewcommand\varphi{\mathfrak{f}}

\renewcommand{\epsilon}{\varepsilon}

\newcommand{\mff}{\mathfrak{f}}

\newcommand{\mfg}{\mathfrak{g}}

\newcommand\PSH{\mathrm{PSH}}
\newcommand\MA{\mathrm{MA}}
\newcommand\FS{\mathrm{FS}}

\newcommand\mi{^{-1}}
\newcommand\vol{\mathrm{vol}}

\newcommand\bsni{\bigskip\noindent}

\newcommand\bbc{\mathbb{C}}

\newcommand\bbr{\mathbb{R}}

\newcommand\cA{\mathcal{A}}

\newcommand\cC{\mathcal{C}}

\newcommand\cE{\mathcal{E}}
\newcommand\cF{\mathcal{F}}

\newcommand\cN{\mathcal{N}}

\newcommand{\norm}[1][\cdot]{\left\|#1\right\|}

\newcommand\NA{^{\mathrm{NA}}}
\newcommand\vbar{{\underline v}}

\newcommand\abar{{\underline a}}

\title[Infinite-dimensional flats in the space of Kähler metrics]{Infinite-dimensional flats in the space of positive metrics on an ample line bundle}

\author[R\'emi Reboulet]{R\'emi Reboulet}
\author[David Witt Nyström]{David Witt Nyström}

\address{R\'emi Reboulet, CNRS, Université Claude Bernard Lyon 1, Bâtiment Doyen Jean Braconnier, 69100 Villeurbanne, France}\email{reboulet@math.univ-lyon1.fr}
\address{David Witt Nyström, Department of Mathematical Sciences, Chalmers University of Technology, Chalmers tvärgata 3, Göteborg 412 96, Sweden}\email{wittnyst@chalmers.se}

\begin{document}

\begin{abstract}
We show that any continuous positive metric on an ample line bundle $L$ lies at the apex of many infinite-dimensional Mabuchi-flat cones. More precisely, given any bounded graded filtration $\cF$ of the section ring of $L$, the set of bounded decreasing convex functions on the support of the Duistermaat--Heckman measure of $\cF$ embeds $L^p$-\textit{isometrically} into the space of bounded positive metrics on $L$ with respect to Darvas' $d_p$ distance for all $p\in[1,\infty)$, and in particular with respect to the Mabuchi metric ($p=2$).
\end{abstract}


\maketitle

\tableofcontents

\section*{Introduction.}

We consider the space of bounded positive metrics on an ample line bundle $L$ over a projective manifold $X$. Mabuchi \cite{mabuchi} defined a distance between two \textit{smooth} positive metrics $\phi_0$, $\phi_1$ on $L$, with
$$d_2(\phi_0,\phi_1)^2:=\inf_{\{\gamma_t\}}\int_0^1\int_X |\dot\gamma_t|^2 \MA(\gamma_t)\,dt,$$
where the infimum is taken over smooth paths $\{\gamma_t\}$ joining $\phi_0$ and $\phi_1$. This distance then extends by continuity to bounded positive metrics. 

By work of Chen \cite{chen} and Darvas \cite{darmabuchi}, this distance in fact makes it into a \textit{geodesic} metric space, which is by itself remarkable: indeed, the Mabuchi distance is highly nonlinear, and geodesics are given as solutions to a degenerate complex Monge--Ampère equation \cite{semmes}. 

Darvas \cite{darmabuchi} later constructed the metric completion $\cE^2(X,L)$ of the space of bounded positive metrics, with respect to the Mabuchi distance, whose elements are singular metrics on $L$ satisfying some integrability condition. More generally, he constructed Finsler $L^p$-type distances $d_p$ for each $p\in [1,\infty)$ on the space of bounded positive metrics on $L$, with completions $\cE^p(X,L)$, which are all included in the larger $\cE^1(X,L)$, called the space of \textit{finite energy} (positive) metrics.

Endowed with the Mabuchi distance, the sets $\cE^p(X,L)$ furthermore have \textit{negative curvature}, in the sense of Busemann, i.e.\ the distance function between two geodesics is convex \cite{chencheng}. In this paper, we will be interested in finding \textit{flat} subspaces of $\cE^p(X,L)$, i.e.\ isometric embeddings of (subsets of) a real vector space into it. Geodesics are of course examples of \textit{one-dimensional flats}, but in this paper, we show that this space in fact admits \textit{infinite-dimensional} flats. 

To construct such subspaces, we will use the method of \textit{quantisation}. Recall that the space of continuous positive metrics is closely related to the spaces of Hermitian norms on $H^0(X,kL)$. Indeed, fixing a smooth volume form $dV$ on $X$, a continuous positive metric $\phi$ induces a Hermitian norm $\norm_{k\phi}$ on the space of sections $H^0(X,kL)$ by
$$\norm[s]_{k\phi}^2:=\int_X |s|^2_{k\phi}\,dV.$$
Conversely, a Hermitian norm induces a metric \textit{via} the Fubini--Study map: letting $(s_{k,i})_i$ be an orthonormal basis for $\norm_{k\phi}$, one defines the smooth positive metric
$$\phi_k:=k\mi \log \sum_i |s_{k,i}|^2.$$
It is well-known from results on quantisation \cite{bouche,catlin,tian,zelditch} that, at the limit, these metrics recover $\phi$:
$$\lim_{k\to\infty}\phi_k=\phi.$$
Chen--Sun \cite{chensun} indeed show that the space of positive metrics is a Gromov--Hausdorff limit of these spaces of Hermitian norms.

Such spaces of norms admit many flat subspaces: fixing a basis of $H^0(X,kL)$, the subclass of Hermitian norms admitting it as an orthogonal basis is in fact isometric to $\bbr^{\dim H^0(X,kL)}$. Those subspaces are called \textit{apartments} and fit more generally into the structure of an \textit{Euclidean building}. It has been conjectured for a long time, see Donaldson \cite[Section 6(iii)]{donaldson} and Codogni \cite[Section 7]{codogni}, that these structures should reflect to an \textit{asymptotic} building structure on the space of Kähler metrics. Our result is in this direction.

We now consider a bounded graded filtration $\cF$ of the section ring of $L$ (in the sense of \cite{bchen}), which means that on each graded piece $H^0(X,kL)$, we have an induced decreasing, left-continuous filtration $\cF^\lambda H^0(X,kL)$ which satisfies a \textit{submultiplicativity} property, in that $$\cF^\lambda H^0(X,kL)\cdot\cF^\gamma H^0(X,mL)\subseteq \cF^{\lambda+\gamma}H^0(X,(k+m)L).$$
To such a filtration are associated its \textit{jumping numbers} $\lambda_{i,k}$, which are the (ordered) values of $\lambda$ at which the filtration changes dimension. Such filtrations were shown by the second author \cite{dwnfiltr} to arise from \textit{test configurations} of $(X,L)$, i.e.\ $\bbc^*$-equivariant one-parameter degenerations of $(X,L)$, although they are strictly more general than test configurations \cite{sze}.

Furthermore, to a filtration $\cF$ one can associate its \textit{Duistermaat--Heckman measure} $\sigma(\cF)$ \cite{bchen}, which is given as the weak limit of the measures $$\sigma_k(\cF):=k^{-n}\sum_i\delta_{\lambda_{i,k}/k}.$$

Let us now pick a convex, decreasing function $\mff$ on $\bbr$. Note that one can always choose an orthogonal basis $(s_{k,i})$ for $\norm_{k\phi}$ which is \textit{adapted} to $\cF$, in that $\cF^{\lambda_i}H^0(X,kL)$ is spanned by $\{s_{k,j}: j\geq i\}$. We then define a \textit{modified} Fubini--Study metric:
$$\phi_k^\mff:=k\mi\log \sum_i |s_{k,i}|^2e^{-k\cdot \mff(\lambda_{k,i}/k)}.$$
We show that the limit
$$\phi^\mff:=\mathrm{usc}\,\lim_{k\to\infty}\phi_k^\mff$$
exists and converges to a bounded positive metric on $L$.

Now, any bounded, convex, decreasing function $\mff$ on the support of the Duistermaat--Heckman measure of $\cF$ may be uniformly approximated by convex, decreasing functions defined on all of $\bbr$, by an extension procedure (see Section \ref{subsect_dh}). We can now extend to this more general case using completeness and an approximation argument. Our main result is then the following:

\bsni\textbf{Main Theorem (\ref{thm_flat1}).} Let $\phi$ be a continuous positive metric on $L$, and let $\cF$ be a bounded graded filtration on $L$. Let $\cC(\cF)$ be the cone of bounded convex decreasing functions on the support of $\sigma(\cF)$. Then, there is an embedding
$$\iota_{\phi,\cF}:\cC(\cF)\hookrightarrow L^\infty\cap\PSH(X,L)$$
mapping $0$ to $\phi$, and such that for all $p\in [1,\infty)$ and any two $\mff,\mfg\in \cC(\cF)$,
    $$d_p(\phi^\mff,\phi^\mfg)=\|f-g\|_{L^p(\sigma(\cF))}.$$

\bsni\textit{Remark 1.} The second statement implies that the map $\iota_{\phi,\cF}$ is an \textit{isometric embedding} of the (infinite-dimensional) cone $\cC(\cF)$ equipped with the $L^p_{\sigma(\cF)}$-metric into the cone of bounded positive metrics on $L$ equipped with the Mabuchi-type Darvas metrics. Therefore, for \textit{any} bounded graded filtration $\cF$, there exists an associated \textit{infinite-dimensional flat cone} with apex $\phi$.

\bsni\textit{Remark 2.} Phong--Sturm \cite{phongsturm} (in the case of test configurations) and Ross--Witt Nyström \cite{rwn} (in general) constructed a geodesic ray, i.e. a \textit{one-dimensional flat cone}, associated to the data of a bounded graded filtration on the section ring of $L$. One recovers those results by taking the family of functions $\mff_t:x\mapsto -t\cdot x$, so that $\{\phi^{\mff_t}\}$ coincides with that geodesic ray. In fact, our construction can be seen as first \textit{modifying} the filtration $\cF$ by the data of the function $\mff$, in a way that is \textit{flat} in a certain sense of the metric geometry of filtrations, and then taking $t=1$ in this geodesic ray construction.

\bsni\textbf{Ideas for the proof.} Let us fix a continuous positive metric $\phi$ on $L$ and a \textit{bounded graded filtration} $\cF$ of the section ring of $L$. While the language of filtrations is likely more familiar to the reader, we prefer to consider instead \textit{bounded graded non-Archimedean norms} $\norm_\bullet\NA$, which are in one-to-one correspondence with filtrations: indeed, on each piece $H^0(X,kL)$ of the section ring, a filtration $\cF^\lambda_k$ induces a non-Archimedean vector space norm by setting
$$\|v\|_k\NA=\sup\{\lambda\in \bbr,\,v\in \cF^\lambda_k H^0(X,kL)\};$$
and this construction can be reversed. Due to the fact that we also work with norms in quantisation, it will make proofs and parallels more clear.

We then consider the class of decreasing convex functions $\mff$ on $\bbr$. The convexity property is an analogue of \textit{submultiplicativity}, and the decreasing property is important to preserve the order of the jumping numbers of the filtration, which become values of the logarithm of the non-Archimedean norm at an orthogonal basis. Associated to it, we construct another bounded graded non-Archimedean norm on the section ring of $L$. The crucial aspect of this article is that the construction associating a new norm to such a function $\mff$ can in fact be made in a \textit{flat} manner, in a certain sense of metric geometry of graded norms, with respect to the $L^p$ norm of $\mff$ restricted to the support of the Duistermaat--Heckman measure of $\cF$. (We note that we technically work with just bounded, decreasing convex functions with bounded right-derivative on this support, which is equivalent, but less convenient for this exposition.) The flatness on the non-Archimedean side, which can be seen as the \textit{boundary at infinity} of the space of positive metrics on $L$, in turn reflects on the space of positive metrics itself, \textit{via} Phong--Sturm's construction. We are greatly helped by recent results of Finski \cite{fin1,fin2} that give algebraic conditions allowing us to understand the metric geometry of positive metrics \textit{via} the metric geometry of norms on the section ring of $L$, which appear e.g.\ in quantisation of geodesic rays. We then extend to general bounded decreasing convex functions on the support of the Duistermaat--Heckman measure by a completeness approximation argument.

\bsni\textbf{Structure of the paper.} The first section gathers results from the literature concerning spaces of norms and quantisation of positive metrics. It is essentially expository.

The second section recalls notions on non-Archimedean norms and Duistermaat--Heckman measures. We explain how to modify a filtration (equivalently a non-Archimedean norm) using the data of a convex decreasing function on the support of the Duistermaat--Heckman measure, then prove our construction yields a new bounded graded norm (Theorem \ref{thm_submult}).

The third section is devoted to the proof of our main theorem. We first study how to modify complex norms, and how this operation preserves submultiplicativity. Finally, we prove the main result (Theorem \ref{thm_flat1}). In the last subsection, we discuss the results and list some open questions, in particular related to Okounkov bodies.

\bsni\textbf{Acknowledgements.} The authors warmly thank Siarhei Finski for many enlightening discussions and comments on a draft version of this article. He in particular suggested that our results should work in a more general setting than the one we initially considered, described in the last section on open questions. They kindly thank the anonymous referee for helpful remarks and comments.

\section{Quantisation of positive metrics.}

Let $X$ be a compact projective complex manifold, endowed with an ample line bundle $L$. The principle of quantisation allows one to construct positive Hermitian \textit{metrics} on $L$ coming from sequences of \textit{norms} on the spaces of sections $H^0(X,kL)$. The ideal type of sequence of norms is a \textit{submultiplicative} sequence of \textit{Hermitian} norms, as submultiplicative sequences have good asymptotic behaviour, while Hermitian norms are easier to analyse, notably because they always admit orthogonal bases.

Unfortunately, in Kähler geometry, most sequences used in quantisation fail to be both submultiplicative and Hermitian. An important class of examples is given by taking $L^2$ or $L^\infty$ norms associated to continuous metrics on the line bundle, as we will explain below. The problem of relating submultiplicativity and Hermitian-ness is at the heart of various problems, including the construction of Mabuchi geodesics \cite{phongsturm,berndtssonprobability,fin1}. Indeed, while one can obtain partial submultiplicativity for sequences of Hermitian norms (coming from the \textit{Bernstein--Markov property} \cite{bb,dwntransforming}), this is not in general sufficient to yield convergence of the quantisation process, and is harder to control under perturbations. Work of Berman--Boucksom \cite{bb} and recently Finski \cite{fin1} provides a way to bridge the gap between those two notions.

\subsection{The space of norms on a complex vector space.}

We begin by recollecting facts about spaces of norms. Throughout this section, $V$ will be a finite-dimensional vector space of dimension $N$ over $\bbc$. We will denote by $\cN(V)$ the space of norms on $V$, and by $\cN^H(V)\subset \cN(V)$ the space of Hermitian norms on $V$.

Those spaces can be endowed with quasi-metric structures. To that end, we consider two norms $\norm$, $\norm'\in\cN(V)$, and define their \textit{successive minima}:
\begin{equation}
\lambda_i(\norm,\norm'):=\sup_{W\in\mathrm{Grass}_i(V)}\inf_{w\in W^\times}\log\frac{\norm[w]'}{\norm[w]},
\end{equation}
$\mathrm{Grass}_i(V)$ denoting the Grassmannian of $i$-dimensional subspaces of $V$. If both norms admit a common orthogonal basis $\underline v=(v_i)$, which is always the case if for example both norms are Hermitian, then up to reordering one has
\begin{equation}
\lambda_i(\norm,\norm')=\log\frac{\norm[v_i]'}{\norm[v_i]},
\end{equation}
corresponding to the logarithms of the eigenvalues of the transition matrix between the two norms. For $p\in [1,\infty)$, one then defines
$$d_p(\norm,\norm')^p:=N\mi \sum_i |\lambda_i(\norm,\norm')|^p,$$
and
$$d_\infty(\norm,\norm'):=\max_i |\lambda_i(\norm,\norm')|.$$
It is then clear that, for $p<q\in[1,\infty]$,
\begin{equation}
d_1(\norm,\norm')\leq d_p(\norm,\norm')\leq d_q(\norm,\norm')\leq d_\infty(\norm,\norm').
\end{equation}

The following fact is well-known for the Hermitian case (e.g. \cite[Theorem 3.1]{be}), and is explained in \cite[Eq (3.9)]{fin2} in the non-Hermitian case:
\begin{proposition}For $p\in[1,\infty]$, $d_p$ defines a metric on $\cN^H(V)$. It satisfies a quasi-triangle inequality on $\cN(V)$ in general: for $\norm$, $\norm'$, $\norm''$,
$$d_p(\norm,\norm')\leq d_p(\norm,\norm'')+d_p(\norm'',\norm')+\log N.$$
\end{proposition}
The spaces $(\cN^H(V),d_p)$ are geodesic for $p\in[1,\infty]$ and uniquely geodesic for $p\in (1,\infty)$. The unique geodesic is given by picking a common orthogonal basis $\underline v$ for two norms $\norm,\norm'$, and defining for $t\in[0,1]$:
$$\norm[\sum c_i v_i]_t:=\sum |c_i|\norm[v_i]^{1-t}\norm[v_i]'{}^t.$$

We also define the \textit{relative volume} of two norms by
$$\vol(\norm,\norm'):=N\mi\sum_i \lambda_i(\norm,\norm').$$
This is a quantity that measures the discrepancy between the volumes of the unit balls of two given norms. It is closely related to the $d_1$ distance: indeed, if $\norm[v]\leq\norm'[v]$ for all $v\in V$, then
$$d_1(\norm,\norm')=\vol(\norm,\norm').$$
One can show that this is a particular case of a more general formula using the \textit{max operator} in the Hermitian case. Given two Hermitian norms diagonalised in a common basis $\vbar$, define $\vee(\norm,\norm')$ as the unique norm diagonalised in the basis $\vbar$ and such that
$$\vee(\norm,\norm')(v_i)=\max(\norm[v_i],\norm[v_i]').$$
A simple computation then shows that
\begin{equation}\label{eq_max}
    d_1(\norm,\norm')=d_1(\norm,\vee(\norm,\norm'))+d_1(\norm',\vee(\norm,\norm')).
\end{equation}

\subsection{Spaces of sequences of norms.}

Let, as in the beginning of this section, $X$ be a compact projective complex manifold endowed with an ample line bundle $L$. For clarity, we will denote $V_k:=H^0(X,kL)$ and $N_k=\dim H^0(X,kL)=h^0(X,kL)$. We now take our interest to sequences $\norm_\bullet:k\mapsto \norm_k$, where $\norm_k$ is a norm on $V_k$.

Let us pick two such sequences $\norm_\bullet$ and $\norm'_\bullet$. Let us define, for $p\in[1,\infty]$,
$$d_p(\norm_\bullet,\norm'_\bullet):=\limsup_k k\mi d_p(\norm_k,\norm'_k).$$
We will say that $\norm_\bullet$ is \textit{$p$-bounded relative to $\norm'_\bullet$} if the limit above is finite. We note that this only defines a pseudo-distance on the set of equivalence classes of relatively $p$-bounded sequences of norms; in general, one defines an equivalence relation whereby two such sequences are \textit{$p$-equivalent} if and only if their $d_p$ distance is zero.

We also define the \textit{relative volume} of $\norm_\bullet$,$\norm'_\bullet$ as
$$\vol(\norm_\bullet,\norm'_\bullet):=\limsup_k k\mi \vol(\norm_k,\norm'_k).$$
A quick computation shows that $\vol$ is $1$-Lipschitz with respect to $d_1$, so that the relative volume is well-defined between two relatively $1$-bounded sequences of norms. Likewise, defining the max operator for sequences of Hermitian norms by
$$\vee(\norm_\bullet,\norm_\bullet'):k\mapsto \vee(\norm_k,\norm'_k)$$
shows that the asymptotic version of \eqref{eq_max} also holds:
\begin{equation}\label{eq_maxasym}
    d_1(\norm_\bullet,\norm'_\bullet)=d_1(\norm_\bullet,\vee(\norm_\bullet,\norm'_\bullet))+d_1(\norm'_\bullet,\vee(\norm_\bullet,\norm'_\bullet)).
\end{equation}

\subsection{Quantisation.}

Given a continuous positive metric $\phi$ on $L$, one can associate to it two different norms on $V_k$: a Hermitian norm $\norm_{k\phi}$, given by
$$\|s\|_{k\phi}^2=\int_X |s|e^{-k\phi} \,dV$$
for a smooth volume form $dV$, and a non-necessarily Hermitian \textit{sup-norm}
$$\|s\|_{k\phi,\infty}:=\sup_X |s|e^{-k\phi}.$$
To any norm $\norm\in\cN(V_k)$, one can associate its \textit{Fubini--Study metric}, a bounded positive metric on $L$, defined at $\ell\in (kL)_x$, $x\in X$, by
$$|\ell|_{\FS(\norm)}:=\inf_{s\in V_k,\,s(x)=\ell}\|s\|.$$
If $\norm$ is furthermore Hermitian, diagonalised in a basis $\vbar_k=(v_{1,k},\dots,v_{N_k,k})$ one recovers \cite[Lemma 2.1]{fin1} the usual definition
\begin{equation}\label{eq:fs}
\FS(\norm)=\log\sum_i |v_{i,k}|^2e^{-2\norm[v_{i,k}]}.
\end{equation}
Classical results on quantisation (historically \cite{bouche,catlin,tian,zelditch} in the smooth case, and e.g.\ \cite{fin1} in the continuous case) show that
$$\lim_k\FS(\norm_{k\phi})=\lim_k\FS(\norm_{k\phi,\infty})=\phi.$$

The space of bounded positive metrics on $L$ can also be endowed with distances closely related to the $d_p$ distances on spaces of norms \cite{darmabuchi}: given $\phi,\psi$ that are positive with respect to a Kähler form $\omega$ representing $c_1(L)$, one defines
$$d_p(\phi,\psi):=\inf_{t\mapsto \gamma_t}\int_0^1\left(\int_X |\dot\gamma_t|^p\,\MA(\gamma_t)\right)^{\frac1p}dt,$$
where the infimum is taken over smooth curves joining the two metrics. The case $p=2$ corresponds to the distance introduced by Mabuchi \cite{mabuchi}. The general case was treated by Darvas \cite{darmabuchi}. We define the \textit{metric completion} of $L^\infty\cap\PSH(X,L)$ with respect to $d_p$ by $\cE^p(X,L)$, following \cite{darmabuchi}; explicitly, they correspond with singular positive metrics $\phi$ on $L$ such that
$$\int_X |\phi|^p \MA(\phi)<\infty.$$
Given $1\leq p \leq q$, we have $\cE^q(X,L)\subseteq \cE^p(X,L)\subseteq \cE^1(X,L)$ in general, the latter space being called the space of \textit{finite energy} (positive) metrics.

As it turns out, by results of Berndtsson \cite{berndtssonprobability}, Darvas--Lu--Rubinstein \cite{dlr}, and Finski \cite[Theorem 1.3, Corollary 3.6]{fin1}, this quantisation procedure can also be metrised:
\begin{theorem}\label{thm_dpquant2}
Let $\phi,\psi$ be bounded positive metrics on $L$. Then,
$$d_p(\phi,\psi)=d_p(\norm_{\bullet\phi},\norm_{\bullet\psi})=d_p(\norm_{\bullet\phi,\infty},\norm_{\bullet\psi,\infty})$$
for $p\in[1,\infty)$, and also for $p=\infty$ if both metrics are continuous (or more generally bounded and regularisable from below, \cite[Proposition 2.15]{fin1}).
\end{theorem}

As explained at the start of the section, one easily sees that the sup-norms associated to a continuous positive metric yield a \textit{submultiplicative} sequence of norms on the algebra of sections of $L$. This means that, given $s_m\in V_m$, $s_n\in V_n$, one has that
$$\norm[s_m\cdot s_n]_{(m+n)\phi,\infty}\leq \norm[s_m]_{m\phi,\infty}\cdot\norm[s_n]_{n\phi,\infty}.$$
More generally, picking a submultiplicative sequence of norms $\norm_\bullet$, the sequence
$(\FS(\norm_k))_k$ is pointwise superadditive. Fekete's lemma (\cite[Lemma 1.2.1]{book:probability}) then ensures that the pointwise limit $\lim_k k\mi \FS(\norm_k)$ exists and equals $\sup_k k\mi \FS(\norm_k)$. We define the Fubini--Study operator to be the usc regularisation of this limit,
$$\FS(\norm_\bullet):=\lim_k{}^*\FS(\norm_k),$$
which is a positive metric, being the usc regularisation of the supremum of a family of positive metrics. To ensure that this operator takes finite values, one must add another condition:
\begin{definition} Let $\norm_\bullet$ be a submultiplicative sequence of norms on the algebra of sections of $L$. We say that $\norm_\bullet$ is a \textit{bounded graded norm} on $R(X,L)$ if there exists a bounded positive metric on $L$ such that
$$d_\infty(\norm_\bullet,\norm_{\bullet\phi,\infty})<\infty.$$
\end{definition}
In particular, as a consequence of the previous theorem and the classification results for sequences of norms from \cite{fin1}, one has
\begin{theorem}[{\cite[Corollary 3.6]{fin1}}]\label{thm_dpquant}If $\norm_\bullet,\norm'_\bullet$ are two bounded graded norms on $L$, then for all $p\in[1,\infty)$,
$$d_p(\FS(\norm_\bullet),\FS(\norm'_\bullet))=d_p(\norm_\bullet,\norm'_\bullet).$$
\end{theorem}
Importantly, this allows us to work directly with general sequences of norms, which do not \textit{a priori} come from bounded metrics on $L$. The case $p=\infty$ is more subtle, but will not be necessary for the purpose of this article.

\subsection{More general sequences of norms.} We explain here how the aforementioned results on quantisation generalise to more abstract norms than those that \textit{a priori} come from positive metrics.

\begin{definition}Let $\norm_\bullet$ be a bounded graded norm. We say that a non-necessarily submultiplicative sequence of norms $\norm_\bullet'$ is \textit{Bernstein--Markov} with respect to $\norm_\bullet$ if, for all $\varepsilon>0$, there exists $C_\varepsilon>1$ such that for all $k$,
$$C_\varepsilon\mi e^{-\varepsilon k}\norm_k \leq \norm'_k \leq C_\varepsilon e^{\varepsilon k}\norm_k,$$
or equivalently
$$d_\infty(\norm_k,\norm'_k)\leq k\varepsilon + \log C_\varepsilon.$$
\end{definition}

\begin{example}[{\cite[Lemma 3.2]{bb}}]
    Given a continuous positive metric $\phi$ on $L$, the sequence of $L^2$-norms associated to $\phi$ (with respect to a smooth volume form) is Bernstein--Markov with respect to the sequence of sup-norms associated to $\phi$.
\end{example}

\begin{proposition}\label{prop_convergence}
    Let $\norm_\bullet$ be a bounded graded norm, and $\norm'_\bullet$ be a sequence of norms that is Bernstein--Markov with respect to $\norm_\bullet$. Then, $\FS(\norm_\bullet)$ exists as a positive bounded metric, and
    $$\FS(\norm_\bullet)=\FS(\norm'_\bullet).$$
\end{proposition}
\begin{proof}
Because $\norm_\bullet$ is submultiplicative, the limit $\lim_k k\mi\FS(\norm_k)$ exists by Fekete's lemma, and equals $\sup_k k\mi\FS(\norm_k)$. Boundedness is easily seen to ensure that this supremum is finite, hence $\FS(\norm_\bullet)$ is the regularised supremum of a bounded above and below family of positive metrics, hence is itself a bounded positive metric. The equality then follows from the fact that $\norm\mapsto k\mi\FS(\norm)$ is obviously $1$-Lipschitz with respect to $d_\infty$ on $\cN^H(H^0(X,kL))$ and the supremum norm on $L^\infty\cap\PSH(X,L)$, while the Bernstein--Markov property implies that $k\mi d_\infty(\norm_k,\norm_k')\to 0$.
\end{proof}
\section{Filtrations and Duistermaat--Heckman measures.}

\subsection{Bounded graded non-Archimedean norms.}\label{subsect_bgn} If $V$ is as before a complex vector space of dimension $N$, one can also consider the space of \textit{non-Archimedean norms} on $V$. Namely, the trivial absolute value $|\cdot|_0$ on $\bbc$, defined by $|z|_0=1$ if and only if $z\neq 0$, is non-Archimedean: trivially, one sees that
$$|z+z'|\leq\max(|z|,|z'|).$$
A non-Archimedean norm on $V$ is a function $\norm\NA:V\to\bbr$ satisfying the usual norm axioms, but with respect to $|\cdot|_0$, and satisfying the \textit{ultrametric} triangle inequality:
$$\norm[v+w]\NA\leq\max(\norm[v]\NA,\norm[w]\NA).$$
An exercise in linear algebra shows that such norms can all be constructed by fixing a basis $\underline v=(v_i)$ of $V$, real numbers $\underline a=(a_i)$, and setting
$$\norm[\sum c_i v_i]\NA:=\max_{i,c_i\neq 0}e^{-a_i}.$$
A basis $\underline v$ in which $\norm\NA$ can be expressed as above is called an \textit{orthogonal basis}, in the non-Archimedean sense. Conversely, any non-Archimedean norm on $V$ admits an orthogonal basis \cite[Example 1.11]{be}.

There is also a trivial norm $\norm_0$ on $V$, defined by $\|v\|_0=1$ if and only if $v\neq 0$. It has the property that any basis of $V$ is an orthogonal basis.

We denote the space of such norms by $\cN\NA(V)$. We note that this space can also be metrised in a similar way as in the complex case, but this will not be required for the purpose of this article. In fact, in a strong sense, $\cN\NA(V)$ with such metric structures can be understood as the boundary at infinity of the negatively curved space $(\cN^H(V),d_p)$.

Returning now to the geometric setting, where $V_k=H^0(X,kL)$ for an ample line bundle on $X$, and $\dim V_k=N_k$, one can also define an appropriate notion of a sequence of norms on the section ring of $L$.

We define a \textit{bounded graded non-Archimedean norm} on $L$ to be a sequence of non-Archimedean norms $\norm_\bullet\NA$, where $\norm_k\NA$ is a non-Archimedean norm on $V_k$, satisfying the following two properties:
\begin{enumerate}
    \item \textit{submultiplicativity:} for all $s_m\in H^0(X,mL)$ and $s_n\in H^0(X,nL)$,
$$\norm[s_m\cdot s_n]^{NA}_{m+n,\varphi}\leq \norm[s_m]\NA_{m,\varphi}\norm[s_n]\NA_{n,\varphi};$$
    \item \textit{boundedness:} there exists $\delta>0$, such that for all $k$,
$$e^{-\delta k}\norm_{0,k}\leq \norm\NA_k\leq e^{\delta k}\norm_{0,k}.$$
Here, $\norm_{0,k}$ is the \textit{trivial norm} on $V_k$.
\end{enumerate}
Bounded graded norms are in bijection with \textit{bounded graded filtrations} of the section ring, in the sense of \cite{dwnfiltr,bchen,bjkstab1}, as in the introduction. In particular, test configurations \cite{bjkstab1} and divisorial valuations \cite{bjkstab2} induce examples of bounded graded norms. The setting of filtrations is more likely to be familiar to the reader, but our results will be more easily stated in terms of non-Archimedean norms, justifying our choice of notation.

\subsection{The Duistermaat--Heckman measure and convex functions.}\label{subsect_dh}

Let $\norm_\bullet\NA$ be a bounded graded non-Archimedean norm. To each $k$, one can define a measure $\sigma_k(\norm_k\NA)$ on the real numbers as follows. Let $(s_{i,k})_{i}$ be an orthogonal basis for $\norm_k\NA$. We then set
$$\sigma_k(\norm_k\NA):=k^{-n}_*\sum_{i=1}^{N_k}\delta_{-k\mi\log\left(\norm[s_{i,k}]\NA\right)}.$$
By \cite{bchen,cmac}, the measures $\sigma_k(\norm_k\NA)$ converge weakly to a measure $\sigma(\norm_\bullet\NA)$ supported on a closed interval of $\bbr$, independently of the choice of orthogonal basis for each $\norm_k\NA$. We define the \textit{Duistermaat--Heckman measure} of $\norm_\bullet\NA$ to be this weak limit.

Our main result in this section is the observation that, given a bounded graded non-Archimedean norm and a bounded decreasing convex function $\mff$ on the support of the Duistermaat--Heckman measure, one can construct a different bounded graded non-Archimedean norm. This construction will be essential in the proof of our main theorem.

\bsni\textbf{The construction.} Fix $\norm\NA_\bullet$, and $(s_{i,k})_i$ a sequence of orthogonal bases as above. Let $\mff$ be a bounded convex, decreasing function on the support of $\sigma(\norm_\bullet\NA)$, with bounded right-derivative.

Define now, for each $k$, a non-Archimedean norm $\norm\NA_{k,\mff}$ as the unique norm satisfying the following two properties:
\begin{enumerate}
    \item $\norm\NA_{k,\mff}$ admits $(s_{i,k})_i$ as an orthogonal basis;
    \item for each $i=1,\dots,N_k$, one has
    $$\norm[s_{i,k}]\NA_{k,\mff}=e^{k\cdot f\left(\frac{-\log\norm[s_{i,k}]\NA_k}{k}\right)}.$$
\end{enumerate}

One needs, as pointed out by an anonymous referee whom we thank for this observation, that  $k\mi \log\norm[s_{i,k}]\NA_k$ lies in the support $[a,b]\subset \bbr$ of the Duistermaat--Heckman measure. Although this is true for finitely generated graded norms, or those induced by divisorial valuations, it is possible that some of those values are smaller than $a$, as follows from \cite[Corollary 5.4, Remark 5.5]{bhj}. By those same results, since $\norm_\bullet\NA$ is bounded, all the $k\mi \log\norm[s_{i,k}]\NA_k$ lie in a segment $[-C,b]$ for a constant $C>0$ depending only on $\norm_\bullet\NA$.

However, we can canonically extend $\mff$ over $[-C,b]$ to a decreasing, convex, bounded function, where for $t<a$ it will be affine with slope the right-derivative of $\mff$ at $a$ (which is finite by assumption). This ensures the construction above to be well-defined, and the values of the extension of $\mff$ outside the support will not matter, as the asymptotics of the norms concentrate on the support $[a,b]$ of $\sigma(\norm_\bullet^{NA})$; in particular, any choice of a decreasing, bounded, convex extension to $[-C,b]$ will suit. We will always implicitly perform this procedure in what follows, denoting by $\mff$ again this choice of an extension. To extend to general $\mff$ with unbounded right derivative, we will use an approximation procedure in the proof of the main theorem.

\begin{example}If $\mff\equiv 0$, then for each $k$, $\norm_{k,\mff}\NA$ is the trivial norm, equal to $1$ on $V_k-\{0\}$. If $\mff(x)=-x$, then $\norm_{\bullet,\mff}\NA=\norm_{\bullet}\NA$.
\end{example}

\begin{theorem}\label{thm_submult}
    The sequence $\norm_{\bullet,\mff}\NA$ is a bounded graded non-Archimedean norm.
\end{theorem}
\begin{proof}
    It is easy to see that, if $\norm_\bullet\NA$ is bounded, then its modification by $\mff$ is also bounded, with constant $C=\sup \mff$. We thus only have to show submultiplicativity: for $s_m\in H^0(X,mL)$ and $s_n\in H^0(X,nL)$, we want to prove that
    $$\norm[s_m\cdot s_n]_{m+n,\mff}\NA\leq\norm[s_m]_{m,\mff}\NA\norm[s_n]_{n,\mff}\NA.$$
    To that end, we proceed in a similar manner to \cite[Theorem 2.6.1]{reb2}: we consider first the following special case, where we assume that $s_m$ belongs to an orthogonal basis for $\norm_m\NA$, and $s_n$ to an orthogonal basis for $\norm_n\NA$. Let us write $\alpha=-\log\norm[s_m]\NA_m$, and $\beta=-\log\norm[s_n]\NA_n$, so that having possibly extended $\mff$ slightly as explained before, $\alpha/m$ and $\beta/n$ belong to the definition set of $\mff$. Since $s_m\cdot s_n\in H^0(X,(m+n)L)$, there is a decomposition
    $$s_m\cdot s_n=\sum a_j s_{j,m+n}$$
    according to the fixed orthogonal basis $(s_{j,m+n})$ for $\norm\NA_{m+n}$. Let us write
    $$\gamma:=-\log\norm[s_m\cdot s_n]_{m+n}\NA.$$
    Since for all $j$ with $s_j\neq 0$ we have 
    $$\norm[s_{j,m+n}]_{m+n}\NA\leq \max_{a_j\neq 0}\norm[s_{j,m+n}]_{m+n}\NA=\norm[s_m\cdot s_n]_{m+n}\NA,$$
    then for all $\delta_j=-\log \norm[s_{j,m+n}]_{m+n}\NA$ with $a_j\neq 0$ we have
    $$\gamma\leq \delta_j,$$
    and furthermore since $\norm\NA_\bullet$ is submultiplicative, one has
    \begin{equation}\label{eq_1}
        \alpha+\beta\leq\gamma\leq \delta_j.
    \end{equation}
    Note now that
    $$\frac{\alpha+\beta}{m+n}=\frac{m}{m+n}\frac{\alpha}{m}+\frac{n}{m+n}\frac{\beta}{n}.$$
    Convexity of $\mff$ then yields
    $$\mff\left(\frac{\alpha+\beta}{m+n}\right)\leq \frac{m}{m+n}\mff\left(\frac{\alpha}{m}\right)+\frac{n}{m+n}\mff\left(\frac{\beta}{n}\right),$$
    so that
    \begin{equation}\label{eq_conv}
        (m+n)\mff\left(\frac{\alpha+\beta}{m+n}\right)\leq m\mff\left(\frac{\alpha}{m}\right)+n\mff\left(\frac{\beta}{n}\right).
    \end{equation}
    With those results in hand, we can prove submultiplicativity in this first special case. Indeed,
    \begin{align*}
        \norm[s_m\cdot s_n]_{m+n}\NA&=\max_{a_j\neq 0}e^{(m+n)\mff\left(\frac{\delta_j}{m+n}\right)}.
    \end{align*}
    By \eqref{eq_1}, and the fact that $\mff$ is decreasing, we have in particular that
    $$\norm[s_m\cdot s_n]_{m+n}\NA\leq e^{(m+n)\mff\left(\frac{\alpha+\beta}{m+n}\right)}.$$
    Thus, by \eqref{eq_conv},
    \begin{align*}
        \norm[s_m\cdot s_n]_{m+n}\NA&\leq e^{m\mff\left(\frac{\alpha}{m}\right)}e^{n\mff\left(\frac{\beta}{n}\right)}\\
        &=\norm[s_m]_{m,\mff}\NA\norm[s_n]_{n,\mff}\NA
    \end{align*}
    by definition.

    We now show submultiplicativity in general. Let $s_m$, $s_n$ be arbitrary sections, which we decompose as
    $$s_m=\sum_{i}c_i s_{i,m},\,s_n=\sum_j c_j s_{j,n}$$
    in orthogonal bases $(s_{i,m})_i$ for $\norm_m^{NA}$ and $(s_{j,n})_j$ for $\norm_n^{NA}$. We then have that
    \begin{align*}
    \norm[s_m\cdot s_n]^{NA}_{m+n,\mff}&\leq \max_{i,c_i\neq 0;j,c_j\neq 0}\norm[s_{i,m}\cdot s_{j,n}]^{NA}_{m+n,\mff}\\
    &\leq \max_{i,c_i\neq 0;j,c_j\neq 0}\norm[s_{i,m}]_{m,\mff}^{NA}\cdot\norm[s_{j,n}]^{NA}_{n,\mff}\\
    &\leq \left(\max_{i,c_i\neq 0}\norm[s_{i,m}]^{NA}_{m,\mff}\right)\cdot\left(\max_{j,c_j\neq 0}\norm[s_{j,n}]^{NA}_{n,\mff}\right)\\
    &=\norm[s_m]^{NA}_{m,\mff}\norm[s_n]^{NA}_{n,\mff}.
\end{align*}
The first inequality follows from the non-Archimedean triangle inequality. The second inequality follows from the result proven above, given that each $s_{i,m}$ and $s_{j,n}$ belong to the relevant orthogonal bases. The last equality is then definitional, concluding the proof.
\end{proof}

The flat cone we will embed in the space of positive metrics is therefore the space of bounded, decreasing convex functions on the support of the Duistermaat--Heckman measure of a bounded graded norm. It is easily seen to be infinite-dimensional, and flat since $L^p$ with $p>1$ is a strictly convex norm.

\section{Embedding the flat cone.}

\subsection{Rescaling by filtrations, I: the finite-dimensional case.}

Let $V$ be a finite-dimensional complex vector space, and $\underline v$ be a fixed basis of $V$. We define the \textit{apartment} $\cA_\vbar\subset\cN^H(V)$ to be the set of all Hermitian norms admitting $\vbar$ as an orthogonal basis. For $p\in[1,\infty]$, this is easily seen to be isometric to $(\bbr^N,L^p)$ \textit{via} the following construction: one fixes a Hermitian norm $\norm\in\cA_\vbar$, and one defines for $\abar\in\bbr^N$
$$\iota_\vbar(\norm,\abar)$$
to be the unique Hermitian norm in $\cA_\vbar$ such that, for all $i$,
$$\iota_\vbar(\norm,\abar)(v_i)=\norm[v_i]e^{-a_i}.$$
This defines a rescaling map
$$\iota_\vbar(\cdot,\cdot):\cA_\vbar\times\bbr^N\to\cA_\vbar.$$

One would like to extend this definition to non-Hermitian norms. Of course, one cannot in general pick an arbitrary diagonalising basis (as those may not exist in general) for this operation. One instead uses an \textit{envelope}-type construction for rescaling. To that end, we construct a \textit{gauge} non-Archimedean norm as follows. Having fixed a basis $\underline v$, and a vector $\abar$, let
$$\norm^{NA}_{\vbar,\underline a}:V\to\bbr_{>0}$$
be the non-Archimedean norm given by
$$\norm[\sum \alpha_i v_i]^{NA}_{\underline a}=\max_{i,\alpha_i\neq 0}e^{-a_i}.$$
We then define the rescaling of $\norm$ \textit{via}
\begin{equation}\iota_\vbar^1(\norm,\abar)(w):=\inf_{w=\sum w_i}\left\{\sum_i \norm[w_i]\norm[w_i]^{NA}_{\vbar,\underline a}\right\},\end{equation}
where the infimum ranges over all decompositions of $w$ as a sum of vectors $w_i\in V$. 

Note that this indeed defines a norm: it is clearly positive definite; $|\cdot|$-homogeneity comes from $|\cdot|$-homogeneity of $\norm$ and $|\cdot|_0$-homogeneity of $\norm^{NA}$; and the triangle inequality from the fact that the data of a decomposition of $v$ and of a decomposition of $w$ gives a decomposition of $v+w$ for any $v,w\in V$.

This gives a rescaling map
$$\iota_{\underline v}^1(\cdot,\cdot):\cN(V)\times\bbr^N\to\cN(V).$$
More generally, we will occasionally denote by
$$\iota^1(\cdot,\cdot):\cN(V)\times\cN\NA(V)\to \cN(V)$$
be the map sending a norm to its envelope by a non-Archimedean norm as defined above. In this case,
$$\iota^1(\norm,\norm\NA_{\vbar,\abar})=\iota^1_\vbar(\norm,\abar).$$
Note that, if $\norm$ is Hermitian, then there always exists a basis $\underline v$ which is orthogonal for $\norm$ and $\norm\NA$. Defining a vector $\underline a$ with components $a_i=-\log\norm[v_i]\NA$, one then has that
$$\iota^1(\norm,\norm\NA)=\iota^1(\norm,\underline a)=\iota(\norm,\underline a).$$
Importantly, the vector $\underline a$ depends on the choice of such a jointly orthogonal basis.

The following result, which follows from the case $p=\infty$ proven in \cite[Lemma 2.8]{fin1}, shows that these rescaling operations have a very nicely controlled metric distortion. In our applications, where we will consider sequences of vector spaces $V_k$ whose dimensions will grow at most as $k^n$ for some fixed $n$, this will imply that this distortion is only growing as $o(k)$.
\begin{lemma}\label{lem_dpcomp}
Let $\underline v$ be a basis of $V$, $\abar\in\bbr^N$, and let $\norm\in\cA_{\underline v}(V)$, $\norm'\in\cN(V)$. Then, for $p\in[1,\infty]$,
$$d_p(\iota_{\underline v}(\norm,\abar),\iota_{\underline v}^1(\norm',\abar))\leq d_\infty(\norm,\norm')+\log\dim V.$$
\end{lemma}
\begin{proof}
For $t\in[0,\infty)$, let $\norm_t\in\cA_\vbar(V)$ be given by
$$\norm[v_i]_t:=\norm[v_i]e^{-ta_i}.$$
In particular, $\norm_1=\iota_\vbar(\norm,\abar)$. Applying \cite[Lemma 2.8]{fin1} at $t=1$ then gives
$$d_\infty(\iota_{\underline v}(\norm,\abar),\iota_{\underline v}^1(\norm',\abar))\leq d_\infty(\norm,\norm')+\log\dim V,$$
and the result follows since $d_p\leq d_\infty$.
\end{proof}

\subsection{Rescaling by filtrations, II: sequences of norms.}

Consider now $\phi$ a continuous positive metric on an ample line bundle $L$ over $X$. Let $\norm_\bullet\NA$ be a bounded graded norm on $L$. Then, it is well-known that for each $k$, one can find a basis $\underline s_k:=(s_{i,k})_i$ of $V_k=H^0(X,kL)$ which is both orthogonal in the Hermitian sense for $\norm_{k\phi}$ and in the non-Archimedean sense for $\norm_k\NA$. Let $\mff$ be a bounded, bounded right-derivative, decreasing convex function on the support of the Duistermaat--Heckman measure $\sigma(\norm_\bullet\NA)$, which gives a bounded graded norm $\norm_{\bullet,\mathfrak{f}}\NA$ as in Theorem \ref{thm_submult}. We then define the two following modifications:
\begin{enumerate}
    \item $\iota(\norm_{\bullet\phi},\mathfrak{f})$ is the sequence of Hermitian norms given in degree $k$ by
    $$\iota(\norm_{k\phi},\mathfrak{f}):=\iota_{\underline s_k}(\norm_{k\phi},\norm_{k,\mathfrak{f}}\NA),$$
    which can be described as the unique Hermitian norm on $V_k$ admitting $\underline s_k$ as an orthogonal basis and such that for all $i$,
    \begin{equation}\label{eq:normdesc}
    \iota(\norm_{k\phi},\mathfrak{f})(s_{i,k})=\norm[s_{i,k}]_{k\phi}e^{k\mff(-k\mi\log\norm[s_{i,k}]_k\NA)};
    \end{equation}
    \item $\iota^1(\norm_{\bullet\phi,\infty},\mathfrak{f})$ is the sequence of norms given in degree $k$ by
    $$\iota^1(\norm_{k\phi,\infty},\mathfrak{f}):=\iota^1(\norm_{k\phi,\infty},\norm_{k,\mathfrak{f}}\NA).$$
\end{enumerate}

\begin{theorem}\label{thm_bm}
    In the notations above,
    \begin{enumerate}
        \item $\iota^1(\norm_{\bullet\phi,\infty},\mathfrak{f})$ is bounded and submultiplicative;
        \item $\iota(\norm_{\bullet\phi},\mathfrak{f})$ is Bernstein--Markov with respect to $\iota^1(\norm_{\bullet\phi,\infty},\mathfrak{f})$.
    \end{enumerate}
\end{theorem}
\begin{proof}
    \textit{Proof of (1).} The submultiplicativity statement appears in the second paragraph \cite[Section 5.1]{fin1}. We give here a proof for the convenience of the reader. 
    Let $s_m\in V_m$ and $s_n\in V_n$. We have by definition
    \begin{align*}
    \iota^1(\norm_{m+n},\norm_{m+n}\NA)(s_m\cdot s_n)=\inf_{s_m\cdot s_n=\sum t_p}\norm[t_p]_{m+n}\norm[t_p]_{m+n}\NA.
    \end{align*}
    Now, any decomposition $s_m=\sum u_q\in V_m$, $s_n= \sum v_r\in V_n$ gives a decomposition $s_m\cdot s_n=\sum u_q\cdot v_r\in V_{m+n}$. This implies that
    \begin{align*}
    \iota^1(\norm_{m+n},\norm_{m+n}\NA)(s_m\cdot s_n)&\leq \inf_{s_m=\sum u_q,s_n=\sum v_r}\norm[u_q\cdot v_r]_{m+n}\norm[u_q\cdot v_r]_{m+n}\NA\\
    &\leq \inf_{s_m=\sum u_q,s_n=\sum v_r}\norm[u_q]_m\norm[v_r]_n\norm[u_q]_m\NA\norm[v_r]_n\NA\\
    &=\inf_{s_m=\sum u_q}\left(\norm[u_q]_m\norm[u_q]_m\NA\right)\inf_{s_n=\sum v_r}\left(\norm[v_r]_n\norm[v_r]_n\NA\right)\\
    &=\iota^1(\norm_{m},\norm_{m}\NA)(s_m)\cdot\iota^1(\norm_{n},\norm_{n}\NA)(s_n),
    \end{align*}
    where in the second inequality we have used submultiplicativity of $\norm_\bullet$ and $\norm'_\bullet$. Boundedness easily follows from the fact that $\mathfrak{f}$ satisfies uniform boundedness.

    \bigskip\noindent\textit{Proof of (2).} Since $\norm_{\bullet \phi}$ is Bernstein--Markov with respect to $\norm_{\bullet \phi,\infty}$, the result then follows from applying Lemma \ref{lem_dpcomp}.
    \end{proof}

\begin{corollary}\label{coro_welldef}
    The limit $\FS(\iota(\norm_{\bullet\phi},\mathfrak{f}))=\FS(\iota^1(\norm_{\bullet\phi,\infty},\mathfrak{f}))$ is well-defined as a bounded positive metric on $L$.
\end{corollary}
\begin{proof}
    This follows from Proposition \ref{prop_convergence}.
\end{proof}

\begin{remark}
    It is then clear from the definitions that $\FS(\iota(\norm_{k\phi},\mathfrak{f}))$ corresponds with the metric $\phi_k^\mff$ from the introduction.
\end{remark}

\subsection{Proof of the main result.}

Given $\norm_\bullet\NA$ a bounded graded norm, let $\cC_{\norm_\bullet\NA}$ be the infinite-dimensional cone of bounded decreasing convex functions on the support of the Duistermaat--Heckman measure of $\norm_\bullet\NA$.

\begin{theorem}\label{thm_flat1}For all continuous positive metrics $\phi$ on $L$ and all bounded graded norms $\norm_\bullet\NA$ on $L$, there exists an embedding
\begin{align*}
\iota:(\cC_{\norm_\bullet\NA},L^p(\sigma(\norm_\bullet\NA)))&\to (L^\infty\cap\PSH(X,L),d_p)\\
\mff&\mapsto \FS(\iota(\norm_{\bullet\phi},\mathfrak{f})),
\end{align*}
mapping the apex $\mff\equiv 0$ to $\phi$, which is isometric for all $p\in[1,\infty)$.
\end{theorem}
\begin{proof}
Let us denote for simplicity by $S=[a,b]$ the support of the Duistermaat--Heckman measure of $\norm_\bullet\NA$, which is a closed interval by \cite[Corollary 5.4]{bhj}.
 Let $\mff,\mfg\in \cC_{\norm_\bullet\NA}$, which we first assume to have finite right derivative at $a$.
By Corollary \ref{coro_welldef} and Theorem \ref{thm_dpquant}, for $p\in[1,\infty)$, one has
\begin{equation}
d_p(\FS(\iota(\norm_{\bullet\phi},\mff)),\FS(\iota(\norm_{\bullet\phi},\mfg)))=d_p(\iota(\norm_{\bullet\phi},\mff),\iota(\norm_{\bullet\phi},\mfg)).
\end{equation}
To prove the isometry result, it therefore suffices to show that we have
\begin{equation}\label{eq_dplebesgue}d_p(\iota(\norm_{\bullet\phi},\mff),\iota(\norm_{\bullet\phi},\mfg))=\|\mff-\mfg\|_{L^p(\sigma(\norm_\bullet\NA))}.
\end{equation}
By definition, one has that
\begin{align*}k\mi d_p(\iota(\norm_{k\phi},\mff),\iota(\norm_{k\phi},\mfg))^p&=N_k\mi\sum_{i=1,\dots,N_k}|\mff(-k\mi\log\norm[s_{i,k}]\NA_k)-\mfg(-k\mi\log\norm[s_{i,k}]\NA_k)|^p\\
&=\int_S |\mff-\mfg|^p\,d\sigma_k(\norm\NA_k).
\end{align*}
Indeed, both norms in the left-hand side are diagonalisable in the basis $\underline s_k$, so that their successive minima are
$$\lambda_i(\iota(\norm_{k\phi},\mff),\iota(\norm_{k\phi},\mfg))=\log \frac{\norm[s_{i,k}]_{k\phi}e^{kg(-k\mi\log\norm[s_{i,k}]\NA_k)}}{\norm[s_{k,i}]_{k\phi}e^{kf(-k\mi\log\norm[s_{i,k}]\NA_k)}}.$$
It then easily follows, since $\sigma_k\to \sigma$ weakly, that one obtains in the limit the desired isometry equality \eqref{eq_dplebesgue}. 

Now, we may extend the map $\iota$ isometrically to general functions in $\cC_{\norm_\bullet\NA}$. Let $\mff$ be such a function, which we approximate in $L^p$ norm by a sequence of convex, decreasing, bounded functions $\mff_k$ with bounded right derivative at $a$ (this is feasible \textit{a fortiori} in $L^\infty$ norm by cutting off $\mff$ at $a+\varepsilon$ and extending it affinely over $[a,a+\varepsilon]$ by the segment joining $\mff(a)$ and $\mff(a+\varepsilon)$). Then this sequence is Cauchy, and by the first part of the proof,
$$d_p(\FS(\iota(\norm_{\bullet\phi},\mff_m)),\FS(\iota(\norm_{\bullet\phi},\mff_n)))=\|\mff_m-\mff_n\|_{L^p(\sigma(\norm_\bullet\NA))}.$$
In particular, the sequence $(\FS(\iota(\norm_{\bullet\phi},\mff_k)))_k$ is $d_p$-Cauchy, hence converges to some metric in $\cE^p(X,L)$  by \cite[Theorem 4.17]{darmabuchi}, which we denote for simplicity by $\FS(\iota(\norm_{\bullet\phi},\mff)\in \cE^p(X,L)$. It is easy to see that this is independent of the choice of an approximation. 

Since $\iota$ is order-reversing in $\mff$ (using \eqref{eq:normdesc} and \eqref{eq:fs}), and our approximation procedure described above gives a decreasing sequence $(\mff_k)_k$, it further follows that this choice of a sequence gives an approximation of $\FS(\iota(\norm_{\bullet\phi},\mff)$ from below, hence $\FS(\iota(\norm_{\bullet\phi},\mff)\in L^\infty\cap\PSH(X,L)$. 

Finally, picking now $\mff,\mfg\in \cC_{\norm_\bullet\NA}$ with bounded right-derivative approximations $(\mff_k)_k,(\mfg_k)_k$, we have that
\begin{align*}
	d_p(\FS(\iota(\norm_{\bullet\phi},\mff)),\FS(\iota(\norm_{\bullet\phi},\mfg)))&=\lim_{k\to\infty}d_p(\FS(\iota(\norm_{\bullet\phi},\mff_k)),\FS(\iota(\norm_{\bullet\phi},\mfg_k)))\\
	&=\lim_{k\to\infty}\|\mff_k-\mfg_k\|_{L^p(\sigma(\norm_\bullet\NA))}\\
	&=\|\mff-\mfg\|_{L^p(\sigma(\norm_\bullet\NA))},
\end{align*}
proving the result.

Note that the map in the statement of the theorem is then easily seen to be injective, i.e.\ defines an isometric \textit{embedding}, for if $\mff\neq \mfg$ while their images define the same metric, then
$$\|\mff-\mfg\|_{L^p(\sigma(\norm_\bullet\NA))}=0.$$
\end{proof}

\subsection{Discussion and open questions.}\,

\bsni\textbf{Relation with Okounkov bodies.} The Okounkov body $\Delta(X,L)$ \cite{okounkov,lazmus,kkh} is a convex body in $\bbr^n$ associated to the algebra of sections of $L$, which captures the volume of $L$. Its definition requires a flag, in particular a hypersurface $Y_1\subset X$. Such a hypersurface induces a bounded graded filtration $\cF_1$ (\textit{via} its order of vanishing valuation). In \cite{dwntransforming}, the second author defines an operation sending a bounded positive metric $\phi$ on $L$ to a convex function on $\Delta(X,L)$, its \textit{Chebyshev transform} $c[\phi]$. By our main result, one can show that a decreasing convex function $\mff$ on the support of the measure associated to $\cF_1$ induces a modification $\phi^\mff$ satisfying $c[\phi^\mff]=c[\phi]+\mff$. In particular, geodesics in this cone are characterised by their Chebyshev transform being affine.

In \cite[Theorem 8.1]{rwn2}, the authors show that there exists a partial moment map from the space of positive metrics to the first coordinate of the Okounkov body. Namely, if $H_1$ is the \textit{exhaustion function} in the sense of \cite[Section 8]{rwn2} associated to the first divisor $Y_1$ in the flag used to construct the Okounkov body, then for all continuous positive metrics on $L$, one has
$$(H_1)_*(dd^c\phi)^n = (p_1)_*d\lambda,$$
where $d\lambda$ is the Lebesgue measure, and $p_1$ is the projection onto the first variable. Exhaustion functions can be constructed using Legendre transforms of test curves; such test curves also admit Chebyshev transforms which give functions in $\cC$ under certain singularity assumptions. Finally, $H_1$ can also be obtained \textit{via} a related, but different, quantisation procedure to ours \cite[Theorem 8.3]{rwn2}. It would therefore be interesting to understand better the connection between our results and those of \cite{rwn2}.

\bsni\textbf{The building structure on the space of Kähler metrics.} The apartments we defined in the space of Hermitian norms on a vector space give it the structure of an \textit{Euclidean building}, with the key property that \textit{any two norms share an apartment}: in other words, any two Hermitian norms can be diagonalised by the same basis, and they (and the geodesic joining them) lie in an entire flat isometric to $\bbr^n$. Our results hint at the existence of such an apartment structure in the space of positive metrics, where the apartments are given by sequences of bases diagonalising a filtration, and the flat model is the space bounded decreasing convex functions on the support of the associated Duistermaat--Heckman measure. This suggests the following, seemingly very difficult open queestion:

\bsni\textbf{Question.} Let $\phi_0,\phi_1$ be two continuous positive metrics on $L$. Then, there exists a choice of filtration such that they lie in the image of a same embedding in the sense of our main theorem.

\bigskip This result would then completely realise the space of positive metrics as an \textit{asymptotic} building, suggesting that the results and techniques from the theory of Euclidean buildings can also be applied to it.

\bsni\textbf{Generalisations.} Much recent work has been devoted to extending pluripotential theory and quantisation procedures both to big or prescribed singularity classes, and to the Kähler setting \cite{darvasxia,ddnl23,drwnxz}. Notions of test configurations (hence possibly filtrations) for Kähler manifolds, more or less algebraic, exist in the literature \cite{dervanross,zak,dxz}. It is expected that an adequate formulation of our results also hold for big or prescribed singularity classes, although many necessary tools, such as Finski's results and a satisfactory theory of Duistermaat--Heckman measures, are as of yet missing.

\bsni\textbf{Toric degenerations.} In the toric setting, one can pick a filtration $\cF$ that is not only submultiplicative but \textit{multiplicative}. Thus, the proof of Theorem \ref{thm_submult} shows that one can embed the entire space of convex functions on the support of the Duistermaat--Heckman measure in the space of positive metrics; not just the decreasing ones. Assuming that $(X,L)$ is not toric, but admits a (metrised) toric degeneration, it would be interesting to understand how the phenomenon of obtaining the entire cone of convex functions arises at the limit.

\bibliographystyle{alpha}
\bibliography{bib}

\newcommand{\etalchar}[1]{$^{#1}$}
\begin{thebibliography}{DRN{\etalchar{+}}23}

\bibitem[BB10]{bb}
Robert Berman and S{\'e}bastien Boucksom.
\newblock Growth of balls of holomorphic sections and energy at equilibrium.
\newblock {\em Invent. Math.}, 181(2):337--394, 2010.

\bibitem[BC11]{bchen}
S{\'e}bastien Boucksom and Huayi Chen.
\newblock Okounkov bodies of filtered linear series.
\newblock {\em Compos. Math.}, 147(4):1205--1229, 2011.

\bibitem[BE21]{be}
S{\'e}bastien Boucksom and Dennis Eriksson.
\newblock Spaces of norms, determinant of cohomology and {F}ekete points in
  non-{A}rchimedean geometry.
\newblock {\em Adv. Math.}, 378:107501, 2021.

\bibitem[Ber09]{berndtssonprobability}
Bo~Berndtsson.
\newblock Probability measures related to geodesics in the space of
  {K}{\"a}hler metrics.
\newblock {\em arXiv preprint arXiv:0907.1806}, 2009.

\bibitem[BHJ17]{bhj}
S{\'e}bastien Boucksom, Tomoyuki Hisamoto, and Mattias Jonsson.
\newblock Uniform {K}-stability, {Duistermaat--Heckman} measures and
  singularities of pairs.
\newblock In {\em Ann. Inst. Fourier (Grenoble)}, volume~67, pages 743--841,
  2017.

\bibitem[BJ21]{bjkstab1}
S{\'e}bastien Boucksom and Mattias Jonsson.
\newblock A non-archimedean approach to {K}-stability, {I}: {M}etric geometry
  of spaces of test configurations and valuations.
\newblock {\em To appear in Ann. Inst. Fourier (Grenoble)}, 2021.

\bibitem[BJ23]{bjkstab2}
Sebastien Boucksom and Mattias Jonsson.
\newblock A non-archimedean approach to k-stability, ii: divisorial stability
  and openness.
\newblock {\em J. Reine Angew. Math.}, 2023.

\bibitem[Bou90]{bouche}
Thierry Bouche.
\newblock Convergence de la m\'{e}trique de {F}ubini-{S}tudy d'un fibr\'{e}
  lin\'{e}aire positif.
\newblock {\em Ann. Inst. Fourier (Grenoble)}, 40(1):117--130, 1990.

\bibitem[Cat99]{catlin}
David Catlin.
\newblock The {B}ergman kernel and a theorem of {T}ian.
\newblock In {\em Analysis and geometry in several complex variables ({K}atata,
  1997)}, Trends Math., pages 1--23. Birkh\"{a}user Boston, Boston, MA, 1999.

\bibitem[CC21]{chencheng}
Xiuxiong Chen and Jingrui Cheng.
\newblock On the constant scalar curvature {K}\"{a}hler metrics
  ({II})---{E}xistence results.
\newblock {\em J. Amer. Math. Soc.}, 34(4):937--1009, 2021.

\bibitem[Che00]{chen}
Xiuxiong Chen.
\newblock The space of {K}{\"a}hler metrics.
\newblock {\em J. Differential Geom.}, 56(2):189--234, 2000.

\bibitem[CM15]{cmac}
Huayi Chen and Catriona Maclean.
\newblock Distribution of logarithmic spectra of the equilibrium energy.
\newblock {\em Manuscripta Math.}, 146(3-4):365--394, 2015.

\bibitem[Cod19]{codogni}
Giulio Codogni.
\newblock Tits buildings and {$K$}-stability.
\newblock {\em Proc. Edinb. Math. Soc. (2)}, 62(3):799--815, 2019.

\bibitem[CS12]{chensun}
Xiuxiong Chen and Song Sun.
\newblock Space of {K}\"{a}hler metrics ({V})---{K}\"{a}hler quantization.
\newblock In {\em Metric and differential geometry}, volume 297 of {\em Progr.
  Math.}, pages 19--41. Birkh\"{a}user/Springer, Basel, 2012.

\bibitem[Dar15]{darmabuchi}
Tam{\'a}s Darvas.
\newblock The {M}abuchi geometry of finite energy classes.
\newblock {\em Adv. Math.}, 285:182--219, 2015.

\bibitem[DDNL23]{ddnl23}
Tam{\'a}s Darvas, Eleonora Di~Nezza, and Chinh~H Lu.
\newblock Relative pluripotential theory on compact {K}ähler manifolds.
\newblock {\em arXiv preprint arXiv:2303.11584}, 2023.

\bibitem[DLR20]{dlr}
Tam{\'a}s Darvas, Chinh~H Lu, and Yanir~A Rubinstein.
\newblock Quantization in geometric pluripotential theory.
\newblock {\em Communications on Pure and Applied Mathematics},
  73(5):1100--1138, 2020.

\bibitem[Don99]{donaldson}
S.~K. Donaldson.
\newblock Symmetric spaces, {K}\"{a}hler geometry and {H}amiltonian dynamics.
\newblock In {\em Northern {C}alifornia {S}ymplectic {G}eometry {S}eminar},
  volume 196 of {\em Amer. Math. Soc. Transl. Ser. 2}, pages 13--33. Amer.
  Math. Soc., Providence, RI, 1999.

\bibitem[DR17]{dervanross}
Ruadha\'{\i} Dervan and Julius Ross.
\newblock K-stability for {K}\"{a}hler manifolds.
\newblock {\em Math. Res. Lett.}, 24(3):689--739, 2017.

\bibitem[DRN{\etalchar{+}}23]{drwnxz}
Tam{\'a}s Darvas, R{\'e}mi Reboulet, David~Witt Nystr{\"o}m, Mingchen Xia, and
  Kewei Zhang.
\newblock Transcendental {O}kounkov bodies.
\newblock {\em arXiv preprint arXiv:2309.07584, to appear in J. Differential
  Geom.}, 2023.

\bibitem[DX24]{darvasxia}
Tam{\'a}s Darvas and Mingchen Xia.
\newblock The volume of pseudoeffective line bundles and partial equilibrium.
\newblock {\em Geom. Topol.}, 28(4):1957--1993, 2024.

\bibitem[DXZ25]{dxz}
Tam{\'a}s Darvas, Mingchen Xia, and Kewei Zhang.
\newblock A transcendental approach to non-archimedean metrics of
  pseudoeffective classes.
\newblock {\em Comment. Math. Helv.}, 2025.

\bibitem[Fin22]{fin1}
Siarhei Finski.
\newblock Submultiplicative norms and filtrations on section rings.
\newblock {\em arXiv preprint arXiv:2210.03039}, 2022.

\bibitem[Fin23]{fin2}
Siarhei Finski.
\newblock Geometry at the infinity of the space of positive metrics: test
  configurations, geodesic rays and chordal distances.
\newblock {\em arXiv preprint arXiv:2305.15300, to appear in J. Reine Angew.
  Math.}, 2023.

\bibitem[KK12]{kkh}
Kiumars Kaveh and Askold~G Khovanskii.
\newblock Newton-{O}kounkov bodies, semigroups of integral points, graded
  algebras and intersection theory.
\newblock {\em Ann. of Math.}, pages 925--978, 2012.

\bibitem[LM09]{lazmus}
Robert Lazarsfeld and Mircea Mustaț{\u{a}}.
\newblock Convex bodies associated to linear series.
\newblock In {\em Ann. Sci. ´Ecole Norm. Sup. (4)}, volume~42, pages 783--835,
  2009.

\bibitem[Mab87]{mabuchi}
Toshiki Mabuchi.
\newblock Some symplectic geometry on compact {K}{\"a}hler manifolds {(I)}.
\newblock {\em Osaka J. Math.}, 24(2):227--252, 1987.

\bibitem[Nys12]{dwnfiltr}
David~Witt Nystr{\"o}m.
\newblock Test configurations and okounkov bodies.
\newblock {\em Compos. Math.}, 148(6):1736--1756, 2012.

\bibitem[Nys14]{dwntransforming}
David~Witt Nystr{\"o}m.
\newblock Transforming metrics on a line bundle to the {O}kounkov body.
\newblock {\em Ann. Sci. ´Ecole Norm. Sup. (4)}, 47(6):1111--1161, 2014.

\bibitem[Oko03]{okounkov}
Andrei Okounkov.
\newblock Why would multiplicities be log-concave?
\newblock In {\em The orbit method in geometry and physics}, pages 329--347.
  Springer, 2003.

\bibitem[PS07]{phongsturm}
Duong~H. Phong and Jacob Sturm.
\newblock Test configurations for {K}-stability and geodesic rays.
\newblock {\em J. Symplectic Geom.}, 5(2):221--247, 2007.

\bibitem[Reb22]{reb2}
R\'{e}mi Reboulet.
\newblock Plurisubharmonic geodesics in spaces of non-{A}rchimedean metrics of
  finite energy.
\newblock {\em J. Reine Angew. Math.}, 793:59--103, 2022.

\bibitem[RN14]{rwn}
Julius Ross and David~Witt Nystr{\"o}m.
\newblock Analytic test configurations and geodesic rays.
\newblock {\em J. Symplectic Geom.}, 12(1):125--169, 2014.

\bibitem[RN17]{rwn2}
Julius Ross and David~Witt Nystr{\"o}m.
\newblock Envelopes of positive metrics with prescribed singularities.
\newblock {\em Ann. Fac. Sci. Toulouse, Math. (6)}, 26(3):687--727, 2017.

\bibitem[SD18]{zak}
Zakarias Sj{\"o}str{\"o}m~Dyrefelt.
\newblock K-semistability of {cscK} manifolds with transcendental cohomology
  class.
\newblock {\em J. Geom. Anal.}, 28(4):2927--2960, 2018.

\bibitem[Sem92]{semmes}
Stephen Semmes.
\newblock Complex {M}onge-{A}mp{\`e}re and symplectic manifolds.
\newblock {\em Amer. J. Math.}, pages 495--550, 1992.

\bibitem[Ste97]{book:probability}
J~Michael Steele.
\newblock {\em Probability theory and combinatorial optimization}.
\newblock SIAM, 1997.

\bibitem[Sz{\'e}15]{sze}
G{\'a}bor Sz{\'e}kelyhidi.
\newblock Filtrations and test-configurations.
\newblock {\em Math. Ann.}, 362(1):451--484, 2015.

\bibitem[Tia90]{tian}
Gang Tian.
\newblock On a set of polarized {K}\"{a}hler metrics on algebraic manifolds.
\newblock {\em J. Differential Geom.}, 32(1):99--130, 1990.

\bibitem[Zel98]{zelditch}
Steve Zelditch.
\newblock Szego kernels and a theorem of {T}ian.
\newblock {\em Int. Math. Res. Not.}, (6):317--331, 1998.

\end{thebibliography}

\end{document}